\documentclass[11pt, oneside]{amsart}   	
\usepackage{geometry}                		
\usepackage{graphicx}				
\usepackage{tikz}
\usepackage{amssymb}
\usepackage{hyperref}

\newtheorem{theorem}{Theorem}[section]
\newtheorem{lemma}[theorem]{Lemma}
\newtheorem{cor}[theorem]{Corollary}

\theoremstyle{definition}
\newtheorem{definition}[theorem]{Definition}
\newtheorem{example}[theorem]{Example}

\theoremstyle{remark}

\numberwithin{equation}{section}

\title{Sharp Bounds for Kulli-Basava Indices of Graphs}
\author{Dr. Sanju Vaidya and Dr. Jeff Chang}

\begin{document}

\begin{abstract}

In this paper, we establish formulas and sharp bounds for general Kulli-Basava indices and characterize graphs that attain these bounds. These indices have been shown to possess strong discriminating power for distinguishing nonisomorphic chemical structures. They are based on the edge neighborhood degrees of vertices in a graph. We also establish bounds for several classes of graphs, including triangle- and quadrangle-free graphs and graphs with a prescribed clique number. The formulas and bounds depend on the numbers of vertices and edges, the minimum and maximum edge neighborhood degrees, and the first Zagreb index.

\end{abstract}
\maketitle

\section{Introduction}

Over the last forty years, graph theory has provided powerful tools for developing models such as quantitative structure--property relationship (QSPR) and quantitative structure--activity relationship (QSAR) models, which are used to analyze the structures and properties of chemical compounds \cite{devillers1999topological, basak1996estimation, todeschini2008handbook}. In particular, QSPR/QSAR analyses of certain antiviral drugs have been valuable in the treatment of COVID-19 patients \cite{kirmani2021topological, mondal2022topological}. In molecular graphs of chemical compounds, vertices correspond to atoms, and edges correspond to the bonds between them. A topological index, or connectivity index, is a molecular descriptor based on the molecular graph of a chemical compound. The pioneering work of Harry Wiener \cite{wiener1947structural} inspired many mathematicians and scientists to define additional topological indices and study their applications in fields such as medicine and environmental science.

In 1972, Gutman and Trinajsti\'{c} introduced the Zagreb indices \cite{gutman1972graph}. These indices are based on vertex degrees and are useful for modeling the chemical and biological properties of compounds, as shown by Devillers et al.\ \cite{devillers1999topological}, Basak et al.\ \cite{basak1996estimation}, and Todeschini et al.\ \cite{todeschini2008handbook}. Gutman et al.\ \cite{gutman2007alkanes} and Zhou et al.\ \cite{zhou2008estrada} used Zagreb indices to study bounds for the Estrada index, which is defined in terms of the eigenvalues of the adjacency matrix.

The discriminating power of a topological index is its ability to distinguish between nonisomorphic chemical structures. Two or more isomers of a chemical compound may have the same topological index. Bonchev et al.\ \cite{bonchev1981isomer} defined the mean isomer degeneracy $d$ as the ratio of the number of isomers to the number of distinct values of the topological index. A value of $d=1$ indicates maximal discriminating power. In 2019, Basavanagoud and Jakkannavar \cite{basavanagoud2019kulli} introduced the Kulli-Basava indices and proved that their mean isomer degeneracy for octane isomers is 1. They also developed regression models for several properties of octane isomers using these indices. Later in 2019, Kulli \cite{VK2019} generalized these indices. They are based on the edge neighborhood degree of a vertex, defined as the sum of the degrees of all edges incident to that vertex. Recall that the degree of an edge $e=uv$ in a graph $G$ is $d(u)+d(v)-2$. For further developments, see Kulli \cite{VK2022}, Afridi et al.\ \cite{afridi2023sharp}, Basavanagoud et al.\ \cite{basavanagoud2024new}, and Asim et al.\ \cite{asim2025topological}.

The main research problem is to derive formulas and bounds for the indices described above and to characterize the graphs that attain those bounds. We establish formulas for the general Kulli-Basava index, derive upper and lower bounds, and describe the corresponding extremal graphs. We also establish bounds for several graph classes, including triangle- and quadrangle-free graphs and graphs with a prescribed clique number. These formulas and bounds involve the numbers of vertices and edges, the minimum and maximum edge neighborhood degrees, and the first Zagreb index, which is the sum of the squares of the vertex degrees. Hu, Li, Shi, Xu, and Gutman \cite{hu2005molecular} investigated the zeroth-order general Randi\'{c} index (also called the general Zagreb index) for molecular graphs with maximum vertex degree at most 4. We extend their approach to the general Kulli-Basava index for graphs whose edge neighborhood degrees are at least 3.

Section 2 reviews the definitions of the first Zagreb index and the Kulli-Basava indices. Section 3 establishes formulas for the general Kulli-Basava index, derives upper and lower bounds, characterizes the graphs that attain them, and presents sharp bounds for several special graph classes. Section 4 contains the conclusion. Chemical graph theory offers many open problems because more than 100 topological indices have been introduced. These indices play a crucial role in modeling the properties of chemical compounds, thereby reducing the cost and time required for applications such as drug development.

\section{Review of Kulli-Basava and Zagreb Indices}

We use the notation and terminology introduced by Gutman et al.\ \cite{gutman2007alkanes}, Basavanagoud and Jakkannavar \cite{basavanagoud2019kulli}, and Kulli \cite{VK2019, VK2022, kulli2019multiplicative}.

Throughout the paper, let $G$ be a simple connected graph with vertex set $V(G)$ and edge set $E(G)$, and let $d(u)$ denote the degree of a vertex $u\in V(G)$. The first Zagreb index, introduced by Gutman and Trinajsti\'{c} \cite{gutman1972graph}, is defined as follows.

\begin{definition} The first Zagreb index is defined as $M_1(G)=\sum_{u\in V(G)}d(u)^2.$
\end{definition}

The Zagreb indices have been generalized and studied extensively. In 2019, Basavanagoud and Jakkannavar \cite{basavanagoud2019kulli} introduced the modified first Kulli-Basava index. It is based on the edge neighborhood degree of a vertex, defined as follows.

\begin{definition} Let $G$ be a graph with vertex set $V(G)$. For a vertex $v\in V(G)$, let $S_e(v)$ denote its edge neighborhood degree, defined as the sum of the degrees of all edges incident to $v$, where the degree of an edge $e=uv$ in $G$ is $d_G(e)=d(u)+d(v)-2$.
\end{definition}

\begin{definition} The modified first Kulli-Basava index is defined by
    \[KB_1^*(G)=\sum_{v\in V(G)} S_e(v)^2.\]
\end{definition}

In 2019, Kulli \cite{VK2019} defined the following general Kulli-Basava index of a graph $G$.

\begin{definition} Let $a$ be a nonzero real number such that $a\neq 1$. The general Kulli-Basava index is defined by
    \[KB^a(G)=\sum_{v\in V(G)} S_e(v)^a.\]

\end{definition}

In 2019, Kulli \cite{kulli2019multiplicative} defined the following general multiplicative Kulli-Basava index of a graph $G$.

\begin{definition} Let $a$ be a nonzero real number such that $a\neq 1$. The general multiplicative Kulli-Basava index is defined by
\[KB^a\Pi(G)=\prod_{v\in V(G)} S_e(v)^a.\]

\end{definition}
We will use Jensen's inequality in the following form.

\begin{theorem}[Jensen's inequality]\label{Jensen}
For any positive integer $k$, if $f$ is strictly convex (that is, $f''(x)>0$), then
\[
f\left(\frac{1}{k}\sum_{i=1}^{k}x_i\right)\leq \frac{1}{k}\sum_{i=1}^{k}f(x_i),
\]
with equality if and only if $x_1=x_2=\cdots=x_k$. If $-f$ is strictly convex, then the inequality is reversed.
\end{theorem}

We also use the following result of Zhou \cite{zhou2006note} for triangle- and quadrangle-free graphs.

\begin{theorem}
Let $G$ be a triangle- and quadrangle-free graph with $n$ vertices. Then its first Zagreb index satisfies $M_1\leq n(n-1)$.
\end{theorem}

We also use the following lower bound of Filipovski \cite{filipovski2021new} for the first Zagreb index of a graph with a prescribed clique number. Recall that the clique number of a graph $G$ is the order of a largest complete subgraph of $G$.

\begin{theorem} Let $G$ be a graph with $n$ vertices, $m$ edges, and clique number $w$. Then
\[
M_1\geq 4mn+\frac{n^3}{w^2}-n^3,
\]
with equality if and only if $G$ is a regular graph of degree $k$ and $w=\frac{n}{n-k}$.
\end{theorem}

\section{Formulas and Bounds for Kulli-Basava Indices}

In this section, we establish formulas for the general Kulli-Basava index and use them to derive sharp upper and lower bounds. We first recall the following lemma, proved by Vaidya and Chang \cite{vaidya2024sharp}.

\begin{lemma}\label{Lemma}
Let $p$ and $q$ be positive integers such that $p<q$, and let $\alpha\in\mathbb{R}\setminus\{0,1\}$. Then the following statements hold.
\begin{enumerate}
\item If $\alpha<0$ or $\alpha>1$, then $(p+i)^\alpha-p^\alpha-i\left(\frac{q^\alpha-p^\alpha}{q-p}\right)\leq 0$ for $1\leq i\leq q-p-1$. If $0<\alpha<1$, then the inequality is reversed.

\item If $\alpha<0$ or $\alpha>1$, then $(p+i)^\alpha-p^\alpha-i((p+1)^\alpha-p^\alpha)\geq 0$ for $2\leq i\leq q-p$. If $0<\alpha<1$, then the inequality is reversed.

\end{enumerate}
\end{lemma}

We also need the following lemma.

\begin{lemma}\label{Lemma2}
Let $p$ and $q$ be positive integers such that $p<q$. Then the following statements hold.
\begin{enumerate}
\item $\ln(p+i)-\ln(p)-i\left(\frac{\ln(q)-\ln(p)}{q-p}\right)\geq 0$ for $1\leq i\leq q-p-1$.

\item $\ln(p+i)-\ln(p)-i(\ln(p+1)-\ln(p))\leq 0$ for $2\leq i\leq q-p$.

\end{enumerate}
\end{lemma}
\begin{proof}
The results follow from the fact that the function $f(x)=\frac{\ln(x)-\ln(p)}{x-p}$ is decreasing for $x>p$.
\end{proof}

\begin{theorem}\label{Theorem1}
Let $G$ be a graph with $n\geq 3$ vertices and $m$ edges. Let $n_i$ denote the number of vertices with edge neighborhood degree $i$, and let $\delta$ and $\Delta$ denote the minimum and maximum edge neighborhood degrees, respectively, where $\delta\neq\Delta$. Let $a\in\mathbb{R}\setminus\{0,1\}$ and set $s_a=\frac{\Delta^a-\delta^a}{\Delta-\delta}$. Then the following statements hold.
    \begin{enumerate}
\item 
        \[KB^a(G) = n\delta^a +  (2M_1 - 4m - n\delta)s_a +\sum_{i =1}^{\Delta - \delta - 1}n_{\delta + i} \left[(\delta + i)^a - \delta^a  - is_a\right].\]

If $a<0$ or $a>1$, the coefficients $(\delta+i)^a-\delta^a-is_a$ are nonpositive for $1\leq i\leq\Delta-\delta-1$; if $0<a<1$, they are nonnegative.

\item If $a>1$, then the natural logarithm of the general multiplicative Kulli-Basava index is
\[
\begin{aligned}
\ln\left[\prod_{v\in V(G)}S_e(v)^a\right]
&=a\left[n\ln(\delta)+\frac{(2M_1-4m-n\delta)(\ln\Delta-\ln\delta)}{\Delta-\delta}\right.\\
&\qquad\left.+\sum_{i=1}^{\Delta-\delta-1}n_{\delta+i}\left(\ln(\delta+i)-\ln(\delta)-i\frac{\ln\Delta-\ln\delta}{\Delta-\delta}\right)\right].
\end{aligned}
\]
Moreover, the coefficients of $n_{\delta+i}$ are nonnegative for $1\leq i\leq\Delta-\delta-1$.
\end{enumerate} 
\end{theorem}
\begin{proof}

By part (iii) of Lemma 2.2 in \cite{basavanagoud2019kulli}, $\sum_{i=\delta}^\Delta i n_i=2M_1-4m$. We also have $n=\sum_{i=\delta}^\Delta n_i$. Solving these equations for $n_\Delta$ gives
\[
n_\Delta=\frac{2M_1 - 4m - n\delta - \sum_{i=\delta + 1}^{\Delta  - 1}(i - \delta)n_i}{\Delta - \delta}.
\]

Writing $KB^a(G)$ in terms of $n_i$ for $\delta\leq i\leq\Delta$ and substituting $n_\delta=n-\sum_{i=\delta+1}^\Delta n_i$, we obtain
\[
KB^a(G) = \sum_{i=\delta}^\Delta  i^a n_i = \delta^a\left( n - \sum_{i= \delta + 1}^\Delta n_i\right) + \sum_{i=\delta + 1}^\Delta  i^a n_i = n\delta^a + \sum_{i = \delta + 1}^\Delta  n_i(i^a - \delta^a).
\]
Substituting the expression for $n_\Delta$ into this sum and applying Lemma \ref{Lemma} yields part (1).

For part (2), write the natural logarithm of the general multiplicative Kulli-Basava index in terms of $n_i$, $\delta\leq i\leq\Delta$, and substitute $n_\delta=n-\sum_{i=\delta+1}^\Delta n_i$. This gives
\[
\ln\left[\prod_{v\in V(G)}S_e(v)^a\right]=a\sum_{i=\delta}^\Delta \ln(i)n_i=a\left[n\ln(\delta)+\sum_{i=\delta+1}^\Delta n_i(\ln(i)-\ln(\delta))\right].
\]
Substituting $n_\Delta$ into the sum and applying part (1) of Lemma \ref{Lemma2} completes the proof.

\end{proof}

The following corollary gives upper and lower bounds for the general Kulli-Basava index, depending on the value of $a$. All the bounds are sharp.

\begin{cor}\label{Cor1}
Let $G$ be a graph with $n\geq 3$ vertices and $m$ edges. Let $\delta$ and $\Delta$ denote the minimum and maximum edge neighborhood degrees, respectively, where $\delta\neq\Delta$. Let $a\in\mathbb{R}\setminus\{0,1\}$ and set $s_a=\frac{\Delta^a-\delta^a}{\Delta-\delta}$. Then the following statements hold.

\begin{enumerate}
        \item If $a < 0$ or $a > 1$, the upper bound for the general Kulli-Basava index is $ KB^a(G) \leq n\delta^a +  (2M_1 - 4m - n\delta)s_a$.

    \item If $0<a<1$, the corresponding lower bound is $KB^a(G)\geq n\delta^a+(2M_1-4m-n\delta)s_a$.

    \item If $a>1$, then the lower bound for the natural logarithm of the general multiplicative Kulli-Basava index is 
$\ln\left[\prod_{v\in V(G)}S_e(v)^a\right]\geq a\left[n\ln(\delta)+\frac{(2M_1-4m-n\delta)(\ln\Delta-\ln\delta)}{\Delta-\delta}\right]$.

\end{enumerate}
In all three parts, equality holds for any graph whose vertices have edge neighborhood degree $\delta$ or $\Delta$.
\end{cor}

\begin{proof}
The results follow from Theorem \ref{Theorem1}.
\end{proof}

\begin{example}
For the graph in Figure \ref{fig:1}, $\delta=3$, $\Delta=18$, $m=n=12$, and $M_1=72$. Every vertex has edge neighborhood degree $3$ or $18$, so the graph attains the bounds in parts (1)--(3) of Corollary \ref{Cor1}.
\begin{figure}[h]
\centering
\begin{minipage}{.4\textwidth}
  \centering
  \begin{tikzpicture}[scale=0.3]
\draw[ultra thick] (-3,3) -- (3,3) -- (3,-3)--(-3,-3) -- (-3,3);
\draw[ultra thick] (3,3)--(6,4);
\draw[ultra thick] (3,3)--(4,6);
\draw[ultra thick] (-3,3)--(-4,6);
\draw[ultra thick] (-3,3)--(-6,4);
\draw[ultra thick] (3,-3)--(6,-4);
\draw[ultra thick] (3,-3)--(4,-6);
\draw[ultra thick] (-3,-3)--(-6,-4);
\draw[ultra thick] (-3,-3)--(-4,-6);
\filldraw[black] (3,3) circle (10pt);
\filldraw[black] (-3,3) circle (10pt);
\filldraw[black] (3,-3) circle (10pt);
\filldraw[black] (-3,-3) circle (10pt);
\filldraw[black] (6,4) circle (10pt);
\filldraw[black] (4,6) circle (10pt);
\filldraw[black] (-6,4) circle (10pt);
\filldraw[black] (-4,6) circle (10pt);
\filldraw[black] (6,-4) circle (10pt);
\filldraw[black] (4,-6) circle (10pt);
\filldraw[black] (-6,-4) circle (10pt);
\filldraw[black] (-4,-6) circle (10pt);
\end{tikzpicture}
\caption{A graph attaining the bounds in Corollary \ref{Cor1}.}
  \label{fig:1}
\end{minipage}%

\end{figure}
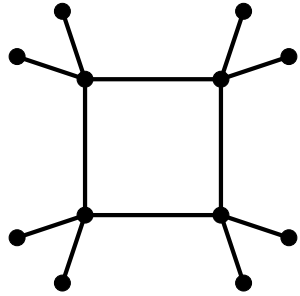
\end{example}

The following corollary gives a sharp upper bound for the Kulli-Basava index of triangle- and quadrangle-free graphs. The bound depends on the numbers of vertices and edges and on the minimum and maximum edge neighborhood degrees.

\begin{cor}
Let $G$ be a triangle- and quadrangle-free graph with $n\geq 3$ vertices and $m$ edges. Let $\delta$ and $\Delta$ denote the minimum and maximum edge neighborhood degrees, respectively, where $\delta\neq\Delta$. Let $a>1$ and set $s_a=\frac{\Delta^a-\delta^a}{\Delta-\delta}$. Then
\[
KB^a(G)\leq n\delta^a+(2n(n-1)-4m-n\delta)s_a,
\]
with equality for a star graph.
\end{cor}

\begin{proof}
Zhou \cite{zhou2006note} proved that the first Zagreb index of a triangle- and quadrangle-free graph satisfies $M_1\leq n(n-1)$. Combining this inequality with Theorem \ref{Theorem1} gives the result.
\end{proof}

The following theorem provides additional formulas for the general and general multiplicative Kulli-Basava indices.

\begin{theorem}\label{Theorem2}
Let $G$ be a graph with $n\geq 3$ vertices and $m$ edges. Let $n_i$ denote the number of vertices with edge neighborhood degree $i$, and let $\delta$ and $\Delta$ denote the minimum and maximum edge neighborhood degrees, respectively. Let $a\in\mathbb{R}\setminus\{0,1\}$. Then the following statements hold.
 \begin{enumerate}
 \item
\begin{align*}
KB^a(G)={}&n\delta^a+(2M_1-4m-n\delta)\bigl((\delta+1)^a-\delta^a\bigr)\\
&+\sum_{i=2}^{\Delta-\delta}n_{\delta+i}
\left[(\delta+i)^a-\delta^a-i\bigl((\delta+1)^a-\delta^a\bigr)\right].
\end{align*}
   
If $a<0$ or $a>1$, the coefficients $(\delta+i)^a-\delta^a-i((\delta+1)^a-\delta^a)$ are nonnegative for $2\leq i\leq\Delta-\delta$; if $0<a<1$, they are nonpositive.

\item If $a>1$, then the natural logarithm of the general multiplicative Kulli-Basava index is
\[
\begin{aligned}
\ln\left[\prod_{v\in V(G)}S_e(v)^a\right]
&=a\left[n\ln(\delta)+(2M_1-4m-n\delta)(\ln(\delta+1)-\ln\delta)\right.\\
&\qquad\left.+\sum_{i=2}^{\Delta-\delta}n_{\delta+i}\left(\ln(\delta+i)-\ln(\delta)-i(\ln(\delta+1)-\ln\delta)\right)\right].
\end{aligned}
\]
Moreover, the coefficients of $n_{\delta+i}$ are nonpositive for $2\leq i\leq\Delta-\delta$.
\end{enumerate} 
 
\end{theorem}

\begin{proof}

By part (iii) of Lemma 2.2 in \cite{basavanagoud2019kulli}, $\sum_{i=\delta}^\Delta i n_i=2M_1-4m$. We also have $n=\sum_{i=\delta}^\Delta n_i$. Hence $n_\delta+n_{\delta+1}=n-\sum_{i=\delta+2}^\Delta n_i$ and $\delta n_\delta+(\delta+1)n_{\delta+1}=2M_1-4m-\sum_{i=\delta+2}^\Delta i n_i$. Solving for $n_\delta$ and $n_{\delta+1}$, substituting the results into $KB^a(G)=\sum_{i=\delta}^\Delta i^a n_i$, and applying Lemma \ref{Lemma} proves part (1). For part (2), write the natural logarithm of the general multiplicative Kulli-Basava index in terms of $n_i$ and substitute $n_\delta=n-\sum_{i=\delta+1}^\Delta n_i$. Then
\[
\ln\left[\prod_{v\in V(G)}S_e(v)^a\right]=a\sum_{i=\delta}^\Delta\ln(i)n_i=a\left[n\ln(\delta)+\sum_{i=\delta+1}^\Delta n_i(\ln(i)-\ln(\delta))\right].
\]
Substituting $n_{\delta+1}$ into the sum and applying part (2) of Lemma \ref{Lemma2} completes the proof.

\end{proof}

The following corollary gives sharp bounds for the general and general multiplicative Kulli-Basava indices, depending on the value of $a$.

\begin{cor}
Let $G$ be a graph with $n\geq 3$ vertices and $m$ edges. Let $\delta$ denote the minimum edge neighborhood degree, and let $a\in\mathbb{R}\setminus\{0,1\}$. Then the following statements hold.

\begin{enumerate}
       
\item If $a < 0$ or $ a > 1$, then the lower bound for the general Kulli-Basava index is $KB^a(G) \geq n\delta^a +  (2M_1 - 4m - n\delta) ((\delta + 1)^a - \delta^a)$.

\item If $0 < a < 1$, then the upper bound for the general Kulli-Basava index is $KB^a(G)  \leq n\delta^a +  (2M_1 - 4m - n\delta) ((\delta + 1)^a - \delta^a)$.

\item If $a>1$, then
\[
\ln\left[\prod_{v\in V(G)}S_e(v)^a\right]\leq a\left[n\ln(\delta)+(2M_1-4m-n\delta)(\ln(\delta+1)-\ln\delta)\right].
\]

In all three parts, equality holds for every regular graph.
    
\end{enumerate}
       
\end{cor}
\begin{proof}
The results follow from Theorem \ref{Theorem2}.
\end{proof}

The following corollary gives a sharp lower bound for graphs with a prescribed clique number.

\begin{cor}
Let $G$ be a graph with $n\geq 3$ vertices, $m$ edges, and clique number $w$. Let $\delta$ denote the minimum edge neighborhood degree, and let $a>1$. Then
\[
KB^a(G)\geq n\delta^a+\left(4m(2n-1)+2n^3\left(\frac{1}{w^2}-1\right)-n\delta\right)((\delta+1)^a-\delta^a),
\]
with equality if $G$ is regular of degree $k$ and $w=\frac{n}{n-k}$.
\end{cor}
\begin{proof}
The result follows from Theorem \ref{Theorem2} and Theorem 2.8 of Filipovski \cite{filipovski2021new}.
\end{proof}

The following theorem gives further bounds for the general Kulli-Basava index.

\begin{theorem}
Let $G$ be a graph with $n\geq 3$ vertices and $m$ edges, and let $a\in\mathbb{R}\setminus\{0,1\}$.

\begin{enumerate}
       
\item If $a < 0$ or $a > 1$, then the lower bound for the general Kulli-Basava index is $KB^a(G) \geq n^{1 - a} (2M_1 - 4m)^a $.

\item If $0<a<1$, then the upper bound  is $KB^a(G)  \leq n^{1 - a}(2M_1 - 4m)^a$.

\end{enumerate}
In both parts, equality holds if $S_e(v)$ is the same for every vertex $v$.
\end{theorem}
\begin{proof}
For part (1), assume that $a<0$ or $a>1$ and let $f(x)=x^a$ for $x>0$. Since $f''(x)>0$, the function is strictly convex. Applying Jensen's inequality (Theorem \ref{Jensen}) to the edge neighborhood degrees and simplifying proves the bound.

For part (2), if $0<a<1$, then $-f(x)=-x^a$ is strictly convex. Applying Jensen's inequality yields the upper bound. Equality holds in both cases when $S_e(v)$ is constant over all vertices $v\in V(G)$.
\end{proof}

The following corollary gives a sharp upper bound for the Kulli-Basava index of triangle-free graphs. The bound depends on the numbers of vertices and edges and on the minimum and maximum vertex degrees.

\begin{cor}
Let $G$ be a triangle-free graph with $n\geq 3$ vertices and $m$ edges. Let $d$ and $D$ denote the minimum and maximum vertex degrees, respectively, where $d\neq D$, and let $a>1$. Then
\[
KB^a(G)\leq(n-2)^a\left[nd^a+(2m-nd)\left(\frac{D^a-d^a}{D-d}\right)\right],
\]
with equality for a star graph.
\end{cor}

\begin{proof}
Because $G$ is triangle-free, every edge $uv\in E(G)$ satisfies $d(u)+d(v)\leq n$. Therefore, the edge neighborhood degree of any vertex $v$ satisfies $S_e(v)\leq(n-2)d(v)$. Hence $KB^a(G)\leq(n-2)^aM_1^a(G)$. Vaidya and Chang \cite{vaidya2024sharp} proved that $M_1^a\leq nd^a+(2m-nd)\left(\frac{D^a-d^a}{D-d}\right)$, and the result follows.
\end{proof}

In Theorem 2.2, Hu, Li, Shi, Xu, and Gutman \cite{hu2005molecular} established formulas for the maximum and minimum zeroth-order general Randi\'{c} indices of $(n,m)$ molecular graphs when $2m-n$ is congruent to $0$, $1$, or $2$ modulo $3$ and the maximum vertex degree is $4$. The following theorem extends this approach and gives sharp bounds for the general Kulli-Basava index when $2M_1-4m-n\delta$ is congruent to an arbitrary integer $r$ modulo $\Delta-\delta$, where $0\leq r<\Delta-\delta$, and the maximum edge neighborhood degree is at least $3$.

\begin{theorem}\label{Theorem3}
Let $G$ be a graph with $n\geq 3$ vertices and $m$ edges. Let $n_i$ denote the number of vertices with edge neighborhood degree $i$, and let $\delta$ and $\Delta$ denote the minimum and maximum edge neighborhood degrees, respectively. Assume that $\Delta-\delta\geq2$ and $2M_1-4m-n\delta=q(\Delta-\delta)+r$, where $q$ is a positive integer and $r$ is an integer such that $0\leq r\leq\Delta-\delta-1$. Let $a\in\mathbb{R}\setminus\{0,1\}$ and set $s_a=\frac{\Delta^a-\delta^a}{\Delta-\delta}$. Then the following statements hold.

 \begin{enumerate}
\item If $r=0$ and $n_\Delta=q$, then $G$ is a bi-edge-neighborhood-degree graph whose vertices have edge neighborhood degrees $\delta$ and $\Delta$.

  \item If $r \geq 1$ and $n_{\Delta} = q$, then $n_i = 0$ for $\delta + r + 1 \leq i \leq \Delta  - 1$ and $n_{\delta + r} \leq 1$.

\item If $r\geq1$, either $a<0$ or $a>1$, and $n_{\delta+r}\neq0$, then
\[
KB^a(G)\leq n\delta^a+(2M_1-4m-n\delta)s_a+(\delta+r)^a-\delta^a-rs_a.
\]

\item If $r\geq1$, $0<a<1$, and $n_{\delta+r}\neq0$, then
\[
KB^a(G)\geq n\delta^a+(2M_1-4m-n\delta)s_a+(\delta+r)^a-\delta^a-rs_a.
\]

\item If $r\geq1$, $a>1$, and $n_{\delta+r}\neq0$, then
\[
\begin{aligned}
\ln\left[\prod_{v\in V(G)}S_e(v)^a\right]\geq
&a\left[n\ln(\delta)+\frac{(2M_1-4m-n\delta)(\ln\Delta-\ln\delta)}{\Delta-\delta}\right.\\
&\qquad\left.+\ln(\delta+r)-\ln(\delta)-r\frac{\ln\Delta-\ln\delta}{\Delta-\delta}\right].
\end{aligned}
\]
 \end{enumerate}
In parts (3)--(5), equality holds if $G$ has $q$ vertices of edge neighborhood degree $\Delta$, one vertex of edge neighborhood degree $\delta+r$, and $n-q-1$ vertices of edge neighborhood degree $\delta$.
 
\end{theorem}

\begin{proof}
By part (iii) of Lemma 2.2 in \cite{basavanagoud2019kulli}, $\sum_{i=\delta}^\Delta i n_i=2M_1-4m$. We also have $n=\sum_{i=\delta}^\Delta n_i$. Solving these equations gives
\[
\sum_{i=\delta+1}^\Delta(i-\delta)n_i=2M_1-4m-n\delta=q(\Delta-\delta)+r.
\]
If $r=0$ and $n_\Delta=q$, then $n_i=0$ for $\delta+1\leq i\leq\Delta-1$, which proves part (1). If $r\geq1$ and $n_\Delta=q$, then $n_i=0$ for $\delta+r+1\leq i\leq\Delta-1$ and $n_{\delta+r}\leq1$, which proves part (2). For part (3), let $r\geq1$ and $n_{\delta+r}\neq0$. By part (1) of Theorem \ref{Theorem1},
$KB^a(G)=n\delta^a +  (2M_1 - 4m - n\delta)s_a +\sum_{i =1}^{\Delta - \delta - 1}n_{\delta + i} \left[(\delta + i)^a - \delta^a  - is_a\right].$

If $a<0$ or $a>1$, then $(\delta+i)^a-\delta^a-is_a\leq0$ for $1\leq i\leq\Delta-\delta-1$. Therefore,

$KB^a(G) \leq n\delta^a +  (2M_1 - 4m -n\delta)s_a + n_{\delta + r} \left[(\delta + r)^a - \delta^a  - (r)s_a\right].$

Since $n_{\delta+r}\neq0$, part (3) follows. Similarly, part (4) follows because $(\delta+i)^a-\delta^a-is_a\geq0$ for $0<a<1$ and $1\leq i\leq\Delta-\delta-1$. Part (5) follows from part (2) of Theorem \ref{Theorem1}. The equality condition in parts (3)--(5) also follows from Theorem \ref{Theorem1}.

\end{proof}

\begin{example}
For the graph in Figure \ref{fig:2}, $n=5$, $m=6$, $\delta=5$, $\Delta=10$, and $M_1=30$. Thus, $2M_1-4m-n\delta=11=2(5)+1$. The graph attains the bounds in parts (3)--(5) of Theorem \ref{Theorem3}.

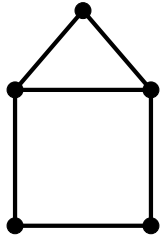
\begin{figure}[h]
\centering
\begin{minipage}{.4\textwidth}
  \centering
  \begin{tikzpicture}[scale=0.3]
\draw[ultra thick] (-3,3) -- (3,3) -- (3,-3)--(-3,-3) -- (-3,3);
\draw[ultra thick] (3,3)--(0,6.5);
\draw[ultra thick] (-3,3)--(0,6.5);
\filldraw[black] (3,3) circle (10pt);
\filldraw[black] (-3,3) circle (10pt);
\filldraw[black] (3,-3) circle (10pt);
\filldraw[black] (-3,-3) circle (10pt);
\filldraw[black] (0,6.5) circle (10pt);
\end{tikzpicture}
\caption{A graph attaining the bounds in Theorem \ref{Theorem3}.}
  \label{fig:2}
\end{minipage}%

\end{figure}

\end{example}

\section{Conclusion}

In this paper, we established formulas for the general Kulli-Basava index of graphs. We also derived upper and lower bounds, characterized graphs that attain them, and obtained sharp bounds for several special graph classes.

Because more than 100 topological indices have been introduced, many open problems remain concerning formulas and bounds for these indices. They play a crucial role in modeling the physicochemical and biological properties of chemical compounds. Such mathematical modeling can reduce both cost and time in applications including drug discovery and chemical hazard assessment.

\bibliographystyle{vancouver}
\bibliography{VancouverExamples.bib}

\end{document}